\documentclass{amsart}
\usepackage{amsmath,amssymb,amsxtra,amsthm,microtype}

\theoremstyle{plain}
\newtheorem{theorem}{Theorem}[section]
\newtheorem{lemma}[theorem]{Lemma}
\newtheorem{proposition}[theorem]{Proposition}
\newtheorem{corollary}[theorem]{Corollary}

\theoremstyle{definition}
\newtheorem{remark}[theorem]{Remark}

\makeatletter
  \def\vhrulefill#1{\leavevmode\leaders\hrule\@height#1\hfill \kern\z@}
\makeatother

\title{An arithmetical excursion via Stoneham numbers}
\date{\today}
\author{Michael Coons}
\address{School of Mathematical and Physical Sciences\\
University of Newcastle\\
University Drive\\
Callaghan NSW 2300\\
Australia}
\email{Michael.Coons@newcastle.edu.au}
\dedicatory{To Professor Peter Borwein on his 60th birthday}
\thanks{The research of M.~Coons is supported by Australian Research Council grant DE140100223.} 


\begin{document}

\begin{abstract}
Let $p$ be a prime and $b$ a primitive root of $p^2$. In this paper, we give an explicit formula for the number of times a value in $\{0,1,\ldots,b-1\}$ occurs in the periodic part of the base-$b$ expansion of $1/p^m$. As a consequence of this result, we prove two recent conjectures of Francisco Arag\'on Artacho, David Bailey, Jonathan Borwein, and Peter Borwein concerning the base-$b$ expansion of Stoneham numbers. 
\end{abstract}

\maketitle

%%%%%%%%%%%%%%%%%%%%%%%%%%%%%%%%%%%%%%%%%%%%%%%%%%%%%%%%
%%%%%%%%%%%%%%%%%%%%%%%%%%%%%%%%%%%%%%%%%%%%%%%%%%%%%%%%

%%%%%%%%%%%%%%%%%%%%%%%%%%%%%%%%%%%%%%%%%%%%%%%%%%%%%%%%%%%%%%%%%%%%%%%%%%%%%
\section{Introduction}
%%%%%%%%%%%%%%%%%%%%%%%%%%%%%%%%%%%%%%%%%%%%%%%%%%%%%%%%%%%%%%%%%%%%%%%%%%%%%

Let $b\geqslant 2$ be an integer. A real number $\alpha\in(0,1)$ is called {\em $b$-normal} if in the base $b$ expansion of $\alpha$ the asymptotic frequency of the occurrence of any word $w\in\{0,1,\ldots,b-1\}^*$ of length $n$ is $1/b^n$. A canonical example of such a number is Champernowne's number, $$\mathcal{C}_{10}:=0.123456789101112131415161718192021\cdots,$$ which given here in base $10$, is the size-ordered concatenation of $\mathbb{N}$ (each number written in base $10$) proceeded by a decimal point. Champernowne's number was shown to be $10$-normal by Champernowne \cite{C1933} in 1933 and transcendental by Mahler \cite{M1937} in 1937. 

In 1973, Stoneham \cite{S1973} defined the following class of numbers. Let $b,c\geqslant 2$ be relatively prime integers. The {\em Stoneham number} $\alpha_{b,c}$ is given by $$\alpha_{b,c}:=\sum_{n\geqslant 1} \frac{1}{c^nb^{c^n}}.$$ Stoneham \cite{S1973} showed that $\alpha_{2,3}$ is $2$-normal. A new proof of this result was given by Bailey and Misiurewicz \cite{BM2006} and finally in 2002, Bailey and Crandall \cite{BC2002} proved that $\alpha_{b,c}$ is $b$-normal for all coprime integers $b,c\geqslant 2$; see also Bailey and Borwein \cite{BB2012}. Transcendence of $\alpha_{b,c}$ follows easily by Mahler's method; the interested reader can see the details Appendix \ref{app}.

Recently Francisco Arag\'on Artacho, David Bailey, Jonathan Borwein, and Peter Borwein~\cite{ABBB} made two conjectures concerning properties of the base-$4$ expansion of the Stoneham number $\alpha_{2,3}$ and the base-$3$ expansion of $\alpha_{3,5}$, respectively. In this paper, we prove their conjectures, and as such they are stated here as theorems (we have fixed a few small typos in their published conjectures).

%\newpage
\begin{theorem}\label{FC1} Let the base-$4$ expansion of $\alpha_{2,3}$ be given by $\alpha_{2,3}:=\sum_{k\geqslant 1}d_k4^{-k}$, with $d_k\in\{0,1,2,3\}$. Then for all $n\geqslant 0$ one has \begin{enumerate}
\item[(i)] $\displaystyle \sum_{k=\frac{3}{2}(3^n+1)}^{\frac{3}{2}(3^n+1)+3^n-1} \left(e^{\frac{\pi i}{2}}\right)^{d_k}=-\begin{cases} i, & \mbox{if $n$ is odd}\\ 1,& \mbox{if $n$ is even},\end{cases}$  
\item[(ii)] $d_k=d_{3^n+k}=d_{2\cdot 3^n+k}$ for $k=\frac{3}{2}(3^n+1),\frac{3}{2}(3^n+1)+1,\ldots,\frac{3}{2}(3^n+1)+3^n-1.$
\end{enumerate}
\end{theorem}

\begin{theorem}\label{FC2} Let the base-$3$ expansion of $\alpha_{3,5}$ be given by $\alpha_{3,5}:=\sum_{k\geqslant 1}a_k3^{-k}$, with $a_k\in\{0,1,2\}$. Then for all $n\geqslant 0$ one has \begin{enumerate}
\item[(i)] $\displaystyle \sum_{k=1+5^{n+1}}^{1+5^{n+1}+4\cdot 5^n} \left(e^{\frac{\pi i}{3}}\right)^{a_k}=(-1)^ne^{\frac{\pi i}{3}}$  
\item[(ii)] $a_k=a_{4\cdot 5^n+k}=a_{8\cdot 5^n+k}=a_{12\cdot 5^n+k}=a_{16\cdot 5^n+k}$ for $k=5^{n+1}+j,$ with $j=1,\ldots,4\cdot 5^n.$
\end{enumerate}
\end{theorem}

We note here that the Stoneham numbers $\alpha_{b,c}$ are in some ways very similar to Champernowne's numbers. They are not concatenations of consecutive integers, but the concatenation of periods of certain rational numbers. Let $b,c\geqslant 2$ be coprime integers and let $w_n$ be the word $w\in\{0,1,\ldots,b-1\}^*$ of minimal length such that $$\left(\frac{1}{c^n}\right)_b =0.\overline{w_n},$$ where $(x)_b$ denotes the base-$b$ expansion of the real number $x$ and $\overline{w}$ denotes the infinitely repeated word $w$. Then the Stoneham numbers are similar to the numbers $$0.w_1w_2w_3w_4w_5\cdots w_n\cdots,$$ which are given by concatenating the words $w_n$. Indeed, the Stoneham number has this structure, but with the $w_j$ repeated and cyclicly shifted.

\begin{remark} While we will be considering the base-$4$ expansion of $\alpha_{2,3}$ we are still dealing with a normal number; $\alpha_{2,3}$ is also $4$-normal. This is given by a result of Schmidt \cite{S1960} who proved in 1960 that the $r$-normal real number $x$ is $s$-normal if $\log r/\log s\in\mathbb{Q}$.\end{remark}

%%%%%%%%%%%%%%%%%%%%%%%%%%%%%%%%%%%%%%%%%%%%%%%%%%%%%%%%%%%%%%%%%%%%%%%%%%%%%
\section{Base-$b$ expansions of rationals}
%%%%%%%%%%%%%%%%%%%%%%%%%%%%%%%%%%%%%%%%%%%%%%%%%%%%%%%%%%%%%%%%%%%%%%%%%%%%%

To prove the above theorems in as much generality as possible we will need to consider how we write a reduced fraction $a/k$ in the base $b$. Such an algorithm is well-known, but we remind the reader here, as it will be useful to have the general framework for the proofs of Theorems \ref{FC1} and \ref{FC2}. To write $a/k$ in the base $b$, we use a sort of modified division algorithm; see Figure \ref{Fig}.

\begin{figure}
\noindent\vhrulefill{1pt}

\noindent\textbf{Base-$b$ Algorithm for $a/k<1$.}

\vspace{-.2cm}
\noindent\vhrulefill{1pt}
\begin{flushleft}

\noindent Let $b,k\geqslant 2$ be integers and $a\geqslant 1$ be an integer coprime to $k$. Set $r_0=a$ and write \begin{align*} 
r_0b&=q_1 k+r_1\\
r_1b&=q_2 k+r_2\\
&\vdots\\
r_{j-1}b&=q_j k+r_j\\
&\vdots
\end{align*} where $q_j\in\{0,1,\ldots,b-1\}$ and $r_j\in\{0,1,\ldots,k-1\}$ for each $j$. Stop when $r_{n}=r_0$. Then $$\left(\frac{a}{k}\right)_b= 0.\overline{q_1q_2\cdots q_{n}}.$$ 
\end{flushleft}

\noindent\vhrulefill{1pt}
\caption{The base-$b$ algorithm for the reduced rational $a/k<1$.}
\label{Fig}
\end{figure}

We record here facts about the base-$b$ algorithm, which we will need.

\begin{lemma}\label{rio} Suppose $b,k\geqslant 2$ are coprime, and that $r_j$ and $q_j$ are defined by the base-$b$ algorithm for $a/k$. Then $\gcd(r_i,k)=1$.
\end{lemma}

\begin{proof} Suppose that $p|k$, and proceed by induction on $i$. Firstly, $r_0=a$ and by assumption $\gcd(r_0,k)=\gcd(a,k)=1.$ 

Now suppose that $\gcd(r_i,k)=1$, so that also $\gcd(r_ib,k)=1$. Then $$r_{i+1}=r_ib-q_{i+1}k\equiv r_ib\not\equiv 0\ (\bmod\ p),$$ since $\gcd(b,k)=1$. Thus $\gcd(r_{i+1},k)=1$.
\end{proof}

Also, we have that equivalent $r_j$ give equal $q_j$.

\begin{lemma} Suppose $b,k\geqslant 2$ are coprime, and that $r_j$ and $q_j$ are defined by the base-$b$ algorithm for the reduced fraction $a/k$. We have $r_i\equiv r_j\ (\bmod\ b)$ if and only if $q_i=q_j$.
\end{lemma}

\begin{proof} Suppose that $r_i\equiv r_j\ (\bmod\ b)$. By considering the difference between $r_{i-1}b=q_i k+r_i$ and $r_{j-1}b=q_j k+r_j$ modulo $b$, we see that $b|(q_i-q_j)k$, so that since $\gcd(b,k)=1$, we have that $b|(q_i-q_j)$. Since $q_i,q_j\in\{0,1,\ldots,b-1\}$, we thus have that $q_i=q_j$. 

Conversely, suppose that $q_i=q_j$. Here, again, we can consider the difference between the defining equations for $q_i$ and $q_j$ modulo $b$; this gives the desired result.
\end{proof}

Indeed, the value of $q_j$ is determined by the residue class of $r_j$ modulo $b$ and the value of $k^{-1}$ modulo $b$.

\begin{lemma}\label{qivalue} Suppose $b,k\geqslant 2$ are coprime, and that $r_j$ and $q_j$ are defined by the base $b$ algorithm for the reduced fraction $a/k$. We have $r_i\equiv j\ (\bmod\ b)$ if and only if $q_i \equiv -jk^{-1} (\bmod b),$ where $q_i\in\{0,1,\ldots,b-1\}$.
\end{lemma}

\begin{proof} If $r_i\equiv j\ (\bmod\ b)$, then the equation $r_{i-1}b=q_ik+r_i$ gives $q_ik\equiv -j\ (\bmod\ b),$ which in turn gives that $q_i\equiv -jk^{-1}\ (\bmod\ b).$ Since $q_i\in[0,b-1]$ we are done with this direction of proof. 

Conversely, suppose that $q_i = (-jk^{-1} \bmod b)$. Then surely, $q_i\equiv -jk^{-1}\ (\bmod\ b)$ and so $q_ik\equiv -j (\bmod\ b).$ Thus, again using $r_{i-1}b=q_ik+r_i$, we have that $r_i\equiv j\ (\bmod\ b)$.
\end{proof}

The following Lemma is a direct corollary of Lemma \ref{qivalue}.

\begin{lemma}\label{qj0} Suppose $b,k\geqslant 2$ are coprime, and that $r_j$ and $q_j$ are defined by the base $b$ algorithm for the reduced fraction $a/k$. We have $r_i\equiv 0\ (\bmod\ b)$ if and only if $q_i=0$.
\end{lemma}

\begin{proof} Apply Lemma \ref{qivalue} with $j=0$.
\end{proof}

We will use the following classical theorem (see \cite[Theorem 12.4]{R2005}) and lemma.

\begin{theorem}\label{period} Let $b$ be a positive integer. Then the base $b$ expansion of a rational number either terminates or is periodic. Further, if $r,s\in\mathbb{Z}$ with $0<r/s<1$ where $\gcd(r,s)=1$ and $s=TU$, where every prime factor of $T$ divides $b$ and $\gcd(U,b)=1$, then the period length of the base-$b$ expansion of $r/s$ is the order of $b$ modulo $U$, and the preperiod length is $N$ where $N$ is the smallest positive integer such that $T|b^N$.
\end{theorem}

Theorem \ref{period} tells us that the base $b$ expansion of $a/k$ is purely periodic (recall for us $\gcd(b,k)=1$), and that the minimal period is ${\rm ord}_k b$, which divides $\varphi(k)$, so that this also is a period. This result can be exploited using the following number-theoretic result, a proof of which can be found in most elementary number theory texts, e.g., see \cite[Theorem 9.10]{R2005}.

\begin{lemma}\label{p2} A primitive root of $p^2$ is a primitive root of $p^k$ for any integer $k\geqslant 2$.
\end{lemma}

Applying Lemma \ref{p2} gives the following result.

\begin{lemma}\label{gs} Let $0<a/p^m<1$ be a rational number in lowest terms and let $b\geqslant 2$ be an integer that is a primitive root of $p^2$. Suppose that $(1/p^m)_b =.\overline{q_1q_2\cdots q_{n}}$ is given by the base $b$ algorithm. Then $$\left(\frac{a}{p^m}\right)_b =.\overline{q_{\sigma(1)}q_{\sigma(2)}\cdots q_{\sigma(n)}}$$ where $\sigma$ is a cyclic shift on $n$ letters.
\end{lemma}

\begin{proof} This is a direct consequence of the base-$b$ algorithm.
\end{proof}

As a consequence of the above lemmas we are able to provide the following characterisation of certain base-$b$ expansions. 

\begin{proposition}\label{zeros} Let $m\geqslant 1$ be an integer, $p$ be an odd prime, $b\geqslant 2$ be an integer coprime to $p$, and $q_j$ and $r_j$ be given by the base-$b$ algorithm for the reduced fraction $a/p^m$. If $b$ is a primitive root of $p$ and $p^2$, then ${\rm period}(a/p^m)=\varphi(p^m)$ and $$\#\left\{j\leqslant \varphi(p^m):q_j=0 \right\}=\left\lfloor\frac{p^m}{b}\right\rfloor-\left\lfloor\frac{p^{m-1}}{b}\right\rfloor.$$ 
\end{proposition}

\begin{proof} The fact that ${\rm period}(a/p^m)_b=\varphi(p^m)$ follows directly from $b$ being a primitive root of $p$ and $p^2$, Lemma \ref{p2} and Theorem \ref{period}. This further implies that the $\varphi(p^m)$ values of $r_i$ given by the base-$b$ algorithm for $a/p^m$ are distinct. Applying Lemma \ref{rio} gives that \begin{equation}\label{sets}\{r_1,r_2,\ldots,r_{\varphi(p^m)}\}=\{i\leqslant p^m:\gcd(i,p)=1\}.\end{equation} Also recall that $$\left(\frac{a}{p^m}\right)_b=.\overline{q_1q_2\cdots q_{\varphi(p^m)}},$$ and that by Lemma \ref{qj0}, $q_i=0$ if and only if $r_i\equiv 0\ (\bmod\ b).$ Note that there are exactly $$\left\lfloor\frac{p^m}{b}\right\rfloor-\left\lfloor\frac{p^m}{bp}\right\rfloor=\left\lfloor\frac{p^m}{b}\right\rfloor-\left\lfloor\frac{p^{m-1}}{b}\right\rfloor$$ elements of $\{i\leqslant p^m:\gcd(i,p)=1\}$ which are divisible by $b$. Thus using the set equality \eqref{sets}, we have that there are exactly $\left\lfloor{p^m}/{b}\right\rfloor-\left\lfloor{p^{m-1}}/{b}\right\rfloor$ elements of $\{r_1,r_2,\ldots,r_{\varphi(p^m)}\}$ divisible by $b$. Appealing to Lemma \ref{qj0} we then have that there are $\left\lfloor{p^m}/{b}\right\rfloor-\left\lfloor{p^{m-1}}/{b}\right\rfloor$ of $q_1,q_2,\ldots,q_{\varphi(p^m)}$ such that $q_j=0$.
\end{proof}

Note that while we record the $q_i=0$ case because of its simplicity, the method can be applied to count any value of $q_i$ that is desired by using the appropriate case of Lemma \ref{qivalue}. In fact, we will do this in a few special cases to prove Theorems \ref{FC1} and \ref{FC2}.

%%%%%%%%%%%%%%%%%%%%%%%%%%%%%%%%%%%%%%%%%%%%%%%%%%%%%%%%%%%%%%%%%%%%%%%%%%%%%
\section{The base-$b$ expansion of the Stoneham number $\alpha_{b,p}$}\label{Baseb}
%%%%%%%%%%%%%%%%%%%%%%%%%%%%%%%%%%%%%%%%%%%%%%%%%%%%%%%%%%%%%%%%%%%%%%%%%%%%%

We will need properties for both the base-$b$ expansion and the base-$b^2$ expansions of the Stoneham number $\alpha_{b,p}$.

\begin{proposition}\label{garat} Let $b,p\geqslant 2$ be coprime integers with $p$ a prime. Denote the base-$b$ expansion of $\alpha_{b,p}$ as $$\alpha_{b,p}=\sum_{j\geqslant 1}\frac{1}{p^jb^{p^j}}=\sum_{k\geqslant 1}\frac{a_k}{b^k},$$ where $a_k\in\{0,1,\ldots,b-1\}$ and write $$\left(\frac{\sum_{j=0}^{m-1} p^j}{p^m}\right)_b=.\overline{q_1q_2\cdots q_{n}}$$ where $q_i$ is determined by the base-$b$ algorithm, for each $i$, so $n={\rm ord}_{p^m}b$. Then $q_i=a_{p^m+jn+i}$ for each $i\in\{1,2,\ldots,n\}$ and each $j\in\{0,1,2,\ldots,\frac{p\cdot\varphi(p^m)}{{\rm ord}_{p^m}b}-1\}$.
\end{proposition}

It is worth noting that Propositions \ref{garat} is the full generalisation of {Theorem~\ref{FC1}(ii)}.

We require the following lemma.

\begin{lemma}\label{agree} Let $b,c\geqslant 2$ be coprime. Then for any $m\geqslant 1$ we have $$\alpha_{b,c}-\sum_{n= 1}^m\frac{1}{c^nb^{c^n}}<\frac{1}{b^{c^{m+1}}}.$$ That is, the base-$b$ expansion of $\alpha_{b,c}$ agrees with the $b$-ary expansion of its $m$-th partial sum up to the $\left({c^{m+1}}\right)$-th place.
\end{lemma}

\begin{proof} Let $m\geqslant 1$ and note that $$\sum_{n\geqslant m+1}\frac{1}{c^n}=\frac{1}{c^{m+1}-c^m}<1.$$ Using this fact, we have that $$\alpha_{b,c}-\sum_{n= 1}^m\frac{1}{c^nb^{c^n}}=\sum_{n\geqslant m+1}\frac{1}{c^nb^{c^n}}<\frac{1}{b^{c^{m+1}}}\sum_{n\geqslant m+1}\frac{1}{c^n}<\frac{1}{b^{c^{m+1}}},$$ which is the desired result.
\end{proof}

\begin{proof}[Proof of Proposition \ref{garat}] Let $m\geqslant 1$, $s_m=p^mb^{p^m}$, and define the positive integer $r_m$ by $$\frac{r_m}{s_m}=\sum_{n= 1}^m\frac{1}{p^nb^{p^n}}.$$ We have then that $$\gcd(r_m,s_m)=\gcd(r_m,p^mb^{p^m})=\gcd(r_m,pb)=1.$$ We apply Theorem \ref{period} with $b=b$, $r=r_m$, $s=s_m$, $T=b^{p^m}$, and $U=p^m$ to give that the period length of the base $b$ expansion of $r_m/s_m$ is the order of $b$ modulo $p^m$, which we will write $${\rm period}(r_m/s_m)={\rm ord}_{p^m}b,$$ and the preperiod length of $r_m/s_m$ is $p^m$, which we will write $${\rm preperiod}(r_m/s_m)=p^m.$$ 

Combining the observations of the previous paragraph with Lemma \ref{agree}, gives that \begin{equation}\label{amp}a_{p^m+1}a_{p^m+2}\ldots a_{p^{m+1}}=\underbrace{www\cdots w}_\text{$\frac{p\cdot\varphi(p^m)}{{\rm ord}_{p^m}b}$ times},\end{equation} where $w=q_1q_2\cdots q_{{\rm ord}_{p^m}b}$ is a word on the alphabet $\{0,1,\ldots,b\}$ with length ${\rm ord}_{p^m}b$. To finish the proof of this proposition, it is enough to appeal to Lemma \ref{agree} to show that $$\left(\frac{\sum_{j=0}^{m-1} p^j}{p^m}\right)_b=.\overline{w},$$ where $w$ is as defined in the previous sentence, which follows directly from the definition of $\alpha_{b,p}.$ 
\end{proof}

Theorem \ref{FC1} concerns a base-$b^2$ expansion; we will provide some specialised results for this case, only when $b=2$, in order to specifically prove Theorem \ref{FC1}, as the more interesting case for generalisations is the base-$b$ case.

\begin{lemma}\label{agree2} Let $b,c\geqslant 2$ be coprime. Then for any $m\geqslant 1$ we have $$\alpha_{b,c}-\sum_{n= 1}^m\frac{1}{c^nb^{c^n}}<\frac{1}{(b^2)^{c^{m+1}/2}}.$$ That is, the base-$b^2$ expansion of $\alpha_{b,c}$ agrees with the base-$b^2$ expansion of its $m$-th partial sum up to the $\lceil{c^{m+1}/2}\rceil$-th place.
\end{lemma}

\begin{proof} This is a direct consequence of Lemma \ref{agree}.
\end{proof}

\begin{proposition}\label{garat2} Let $p$ be an odd prime such that $2$ is a primitive root of $p$ and $p^2$. Denote the base-$4$ expansion of $\alpha_{2,p}$ as $$\alpha_{2,p}=\sum_{j\geqslant 1}\frac{1}{p^j2^{p^j}}=\sum_{k\geqslant 1}\frac{d_k}{4^{k}},$$ where $d_k\in\{0,1,\ldots,3\}$ and write $$\left(\frac{\sum_{j=0}^{m-1} p^j}{p^m}\right)_{4}=.\overline{q_1q_2\cdots q_{n}}$$ where the $q_i$s are determined by the base-$4$ algorithm, so $n={\rm ord}_{p^m}4=\varphi(p^m)/2$. Then $q_i=d_{\frac{p^{m}+1}{2}+jn+i}$ for each $i\in\{1,\ldots,n\}$ and each $j\in\{0,1,2,\ldots,p-1\}$.
\end{proposition}

\begin{proof} This proposition follows as a corollary of Proposition \ref{garat}. Indeed, by Proposition \ref{garat}, we have a prefix $u$ of odd length $p$ and words $w_m$ of even length $\varphi(p^m)$ such that $$(\alpha_{2,p})_2=.u\underbrace{w_1w_1\cdots w_1}_\text{$p$ times}\underbrace{w_2w_2\cdots w_2}_\text{$p$ times}\cdots \underbrace{w_mw_m\cdots w_m}_\text{$p$ times}\cdots.$$ Now the word $w_m$ is the minimal repeated word given by the base-$2$ expansion of $\left({\sum_{j=0}^{m-1} p^j}\right)/{p^m}$. But $$0<\frac{\sum_{j=0}^{m-1} p^j}{p^m}=\frac{p^m-1}{p^m(p-1)}<\frac{1}{p-1}\leqslant \frac{1}{2},$$ and so the first letter of $w_m$, for each $m$, is necessarily $0$. Define the word $v_m$ by $w_m=0v_m$. Then \begin{align}\nonumber (\alpha_{2,p})_2 &=.u\underbrace{w_1w_1\cdots w_1}_\text{$p$ times}\underbrace{w_2w_2\cdots w_2}_\text{$p$ times}\cdots \underbrace{w_mw_m\cdots w_m}_\text{$p$ times}\cdots\\
\nonumber &=.u\underbrace{0v_10v_1\cdots 0v_1}_\text{$p$ times}\underbrace{0v_20v_2\cdots 0v_2}_\text{$p$ times}\cdots \underbrace{0v_m0v_m\cdots 0v_m}_\text{$p$ times}\cdots\\
\label{abuv}&=.u0\underbrace{v_10v_10\cdots v_10}_\text{$p$ times}\underbrace{v_20v_20\cdots v_20}_\text{$p$ times}\cdots \underbrace{v_m0v_m0\cdots v_m0}_\text{$p$ times}\cdots,
\end{align} where we have that the word $u0$ is of even length $p+1$ and the word $v_m0$ is of even length $\varphi(p^m)$. 

As in the statement of Proposition \ref{garat}, let $a_k$ be the $k$th letter in the base-$2$ expansion of $\alpha_{2,p}$, and as in the statement of the current proposition, let $d_k$ be the $k$th letter in the base-$4$ expansion of $\alpha_{2,p}$. Then $$d_k=2a_{2k-1}+a_{2k}.$$ Using this fact, it is an immediate consequence of \eqref{abuv} that there are words $U$ of length $(p+1)/2$ and $W_m$ of length $\varphi(p^m)/2$ such that $$(\alpha_{2,p})_{4}=.U\underbrace{W_1W_1\cdots W_1}_\text{$p$ times}\underbrace{W_2W_2\cdots W_2}_\text{$p$ times}\cdots \underbrace{W_mW_m\cdots W_m}_\text{$p$ times}\cdots.$$

As in Proposition \ref{garat}, to finish the proof of this proposition, it is enough to apply Lemma \ref{agree2} to show that $$\left(\frac{\sum_{j=0}^{m-1} p^j}{p^m}\right)_{4}=.\overline{W_m},$$ where $W_m$ is as defined in the previous sentence, which follows directly from the definition of $\alpha_{2,p}.$
\end{proof}

%%%%%%%%%%%%%%%%%%%%%%%%%%%%%%%%%%%%%%%%%%%%%%%%%%%%%%%%%%%%%%%%%%%%%%%%%%%%%
\section{The Aragon, Borwein, Borwein, and Bailey Conjectures}
%%%%%%%%%%%%%%%%%%%%%%%%%%%%%%%%%%%%%%%%%%%%%%%%%%%%%%%%%%%%%%%%%%%%%%%%%%%%%

In this section, we apply the results of Section \ref{Baseb} to prove Theorems \ref{FC1} and \ref{FC2}. As is turns out, the proof of Theorem \ref{FC2} is a bit more straightforward, so we present its proof first.

\begin{proof}[Proof of Theorem \ref{FC2}] For convenience let us write $\omega:=e^{\pi i/3}$ and let $r_i$ and $q_i$ be given by the base-$3$ algorithm for $1/5^n$. Note that by Proposition \ref{garat}, we have that $$\sum_{k=1+5^{n+1}}^{1+5^{n+1}+4\cdot 5^n} \omega^{a_k}=\sum_{j=0}^2 \#\{i\leqslant \varphi(5^{n+1}):q_i=j\}\cdot \omega^j.$$ 

Now $\#\{i\leqslant \varphi(5^n):q_i=j\}$ can be given by looking at where the number $5^n$ lies modulo $15$. Since, for every $15$ consecutive numbers, $12$ of them are coprime to $5$, and these $12$ fall into the $3$ equivalence classes modulo $3$ with an equal frequency of $4$ times each, we need only look at the remainder of $5^n$ modulo $15$. An easy calculation gives that $$5^n\equiv \begin{cases} 5\ (\bmod\ 15) & \mbox{if $n$ is odd}\\ 10\ (\bmod\ 15) & \mbox{if $n$ is even.}\end{cases}$$ This allows us to give that \begin{align*}\#\{i\leqslant\varphi(5^n): r_i\equiv 0\ (\bmod\ 3)\}&=\begin{cases} 4\cdot\left\lfloor\frac{5^n}{15}\right\rfloor+1 & \mbox{if $n$ is odd}\\ 4\cdot\left\lfloor\frac{5^n}{15}\right\rfloor+3 & \mbox{if $n$ is even},\end{cases}\\
\#\{i\leqslant\varphi(5^n): r_i\equiv 1\ (\bmod\ 3)\}&=\begin{cases} 4\cdot\left\lfloor\frac{5^n}{15}\right\rfloor+2 & \mbox{if $n$ is odd}\\ 4\cdot\left\lfloor\frac{5^n}{15}\right\rfloor+3 & \mbox{if $n$ is even},\end{cases}
\end{align*} and $$\#\{i\leqslant\varphi(5^n): r_i\equiv 2\ (\bmod\ 3)\}=\begin{cases} 4\cdot\left\lfloor\frac{5^n}{15}\right\rfloor+1 & \mbox{if $n$ is odd}\\ 4\cdot\left\lfloor\frac{5^n}{15}\right\rfloor+2 & \mbox{if $n$ is even}.\end{cases}$$ Applying Lemma \ref{qivalue} to the preceding equalities gives that \begin{align*}\#\{i\leqslant\varphi(5^n): q_i= 0\}&=\begin{cases} 4\cdot\left\lfloor\frac{5^n}{15}\right\rfloor+1 & \mbox{if $n$ is odd}\\ 4\cdot\left\lfloor\frac{5^n}{15}\right\rfloor+3 & \mbox{if $n$ is even},\end{cases}\\
\#\{i\leqslant\varphi(5^n): q_i=1\}&=\begin{cases} 4\cdot\left\lfloor\frac{5^n}{15}\right\rfloor+2 & \mbox{if $n$ is odd}\\ 4\cdot\left\lfloor\frac{5^n}{15}\right\rfloor+2 & \mbox{if $n$ is even},\end{cases}
\end{align*} and $$\#\{i\leqslant\varphi(5^n): q_i=2\}=\begin{cases} 4\cdot\left\lfloor\frac{5^n}{15}\right\rfloor+1 & \mbox{if $n$ is odd}\\ 4\cdot\left\lfloor\frac{5^n}{15}\right\rfloor+3 & \mbox{if $n$ is even}.\end{cases}$$ Since $1+\omega+\omega^2=0$, we thus have that \begin{align*}\sum_{k=1+5^{n+1}}^{1+5^{n+1}+4\cdot 5^n} \omega^{a_k}&=\sum_{j=0}^2 \#\{i\leqslant \varphi(5^{n+1}):q_i=j\}\cdot \omega^j\\
&=\begin{cases} \omega & \mbox{if $n+1$ is odd} \\ -\omega & \mbox{if $n+1$ is even} \end{cases}\\
&=(-1)^n\omega,\end{align*} which proves part (i).
 
Part (ii) follows directly from Proposition \ref{garat} with $b=3$ and $p=5$.
\end{proof}

\begin{proof}[Proof of Theorem \ref{FC1}] Note that $$\frac{1}{3^n 2^{3^n}}=\frac{8}{3^n}\cdot\frac{1}{4^{\frac{3}{2}(3^{n-1}+1)}}.$$ Let $r_i$ and $q_i$ be given by the base $4$ algorithm for $8/3^n$. We will use the fact that each of these $r_i$ is equivalent to $2$ modulo 3. This is easily seen as we have for each $i$ that $r_{i-1}4=q_i3^n+r_i$, so that taking this equality modulo $3$ we have that $r_{i-1}\equiv r_i\ (\bmod\ 3)$. Recalling that $r_0=8$ shows that indeed $r_i\equiv 2\ (\bmod\ 3)$ for each $i$.

Since ${\rm ord}_{3^n}4=3^{n-1}$, we have, by Proposition \ref{garat2}, that $$\sum_{k=\frac{3}{2}(3^n+1)}^{\frac{3}{2}(3^n+1)+3^n-1} (e^{\frac{\pi i}{2}})^{a_k}=\sum_{j=0}^3 \#\{i\leqslant \varphi(3^{n+1})/2:q_i=j\}\cdot (e^{\frac{\pi i}{2}})^j.$$ Now $\#\{i\leqslant 3^n:q_i=j\}$ can be given by looking at where the number $3^n$ lies modulo $12$. Since, for every $12$ consecutive numbers, $4$ of them are equivalent to $2$ modulo $3$, and these $4$ fall into the $4$ distinct equivalence classes modulo $4$, we must consider the remainder of $3^n$ modulo $12$. We have that $$3^n\equiv \begin{cases} 3\ (\bmod\ 12) & \mbox{if $n$ is odd}\\ 9\ (\bmod\ 12) & \mbox{if $n$ is even.}\end{cases}$$ Thus we have that \begin{align*}
\#\{i\leqslant \varphi(3^{n})/2: r_i\equiv 0\ (\bmod\ 4)\}&= \begin{cases} \left\lfloor\frac{3^n}{12}\right\rfloor & \mbox{if $n$ is odd}\\ \left\lfloor\frac{3^n}{12}\right\rfloor+1 & \mbox{if $n$ is even},\end{cases}\\
\#\{i\leqslant \varphi(3^{n})/2: r_i\equiv 1\ (\bmod\ 4)\}&= \begin{cases} \left\lfloor\frac{3^n}{12}\right\rfloor & \mbox{if $n$ is odd}\\ \left\lfloor\frac{3^n}{12}\right\rfloor+1 & \mbox{if $n$ is even},\end{cases}\\
\#\{i\leqslant \varphi(3^{n})/2: r_i\equiv 2\ (\bmod\ 4)\}&= \begin{cases} \left\lfloor\frac{3^n}{12}\right\rfloor+1 & \mbox{if $n$ is odd}\\ \left\lfloor\frac{3^n}{12}\right\rfloor+1 & \mbox{if $n$ is even},\end{cases}
\end{align*} and 
$$\#\{i\leqslant \varphi(3^{n})/2: r_i\equiv 3\ (\bmod\ 4)\}= \begin{cases} \left\lfloor\frac{3^n}{12}\right\rfloor & \mbox{if $n$ is odd}\\ \left\lfloor\frac{3^n}{12}\right\rfloor & \mbox{if $n$ is even}.\end{cases}$$

By Lemma \ref{qivalue}, we have that \begin{align*}
\#\{i\leqslant \varphi(3^{n})/2: q_i=0\}&= \begin{cases} \left\lfloor\frac{3^n}{12}\right\rfloor & \mbox{if $n$ is odd}\\ \left\lfloor\frac{3^n}{12}\right\rfloor+1 & \mbox{if $n$ is even},\end{cases}\\
\#\{i\leqslant \varphi(3^{n})/2: q_i= 1\}&= \begin{cases} \left\lfloor\frac{3^n}{12}\right\rfloor & \mbox{if $n$ is odd}\\ \left\lfloor\frac{3^n}{12}\right\rfloor & \mbox{if $n$ is even},\end{cases}\\
\#\{i\leqslant \varphi(3^{n})/2: q_i= 2\}&= \begin{cases} \left\lfloor\frac{3^n}{12}\right\rfloor+1 & \mbox{if $n$ is odd}\\ \left\lfloor\frac{3^n}{12}\right\rfloor+1 & \mbox{if $n$ is even},\end{cases}
\end{align*} and 
$$\#\{i\leqslant \varphi(3^{n})/2: q_i= 3\}= \begin{cases} \left\lfloor\frac{3^n}{12}\right\rfloor & \mbox{if $n$ is odd}\\ \left\lfloor\frac{3^n}{12}\right\rfloor+1 & \mbox{if $n$ is even}.\end{cases}$$

Since $1+(e^{\frac{\pi i}{2}})+(e^{\frac{\pi i}{2}})^2+(e^{\frac{\pi i}{2}})^3=0$, we thus have that \begin{align*}\sum_{k=\frac{3}{2}(3^n+1)}^{\frac{3}{2}(3^n+1)+3^n-1} (e^{\frac{\pi i}{2}})^{a_k}&=\sum_{j=0}^3 \#\{i\leqslant \varphi(3^{n+1})/2:q_i=j\}\cdot (e^{\frac{\pi i}{2}})^j\\
&=\begin{cases} -1 & \mbox{if $n+1$ is odd} \\ -i & \mbox{if $n+1$ is even} \end{cases}\\
&=-\begin{cases} i & \mbox{if $n$ is odd} \\ 1 & \mbox{if $n$ is even}, \end{cases}\end{align*} which proves part (i).

Part (ii) follows directly from Proposition \ref{garat2} with $b=2$ and $p=3$.
\end{proof}

%%%%%%%%%%%%%%%%%%%%%%%%%%%%%%%%%%%%%%%%%%%%%%%%%%%%%%%%%%%%%%%%%%%%%%%%%%%%%
\noindent\textbf{Acknowledgements.} To be written ...
%%%%%%%%%%%%%%%%%%%%%%%%%%%%%%%%%%%%%%%%%%%%%%%%%%%%%%%%%%%%%%%%%%%%%%%%%%%%%

\appendix

%%%%%%%%%%%%%%%%%%%%%%%%%%%%%%%%%%%%%
\section{Transcendence of Stoneham numbers}\label{app}
%%%%%%%%%%%%%%%%%%%%%%%%%%%%%%%%%%%%%

In this appendix, we give details of the transcendence of the Stoneham number $\alpha_{b,c}$ for any choice of integers $b,c\geqslant 2$. In fact, Mahler's method gives much stronger results, which imply this desired conclusion.

We start out by letting $c\geqslant 2$ be an integer and define $$F_c(x):=\sum_{n\geqslant 1}\frac{x^{c^n}}{c^n}.$$ Notice that $F_c(x)$ satisfies the Mahler functional equation \begin{equation}\label{stonefunk} F_c(x^c)=cF_c(x)-x^c.\end{equation} Now suppose that $F_c(x)\in\mathbb{C}(x)$. Then there are polynomials $a(x),b(x)\in\mathbb{C}[x]$ such that $$F_c(x)-\frac{a(x)}{b(x)}=0.$$ Since $F_c(x)\in\mathbb{C}[[x]]$ is not a polynomial, we may assume, without loss of generality, that $\gcd(a(x),b(x))=1$ and $b(0)\neq 0$ and $b(x)\notin\mathbb{C}$. Sending $x\to x^c$ and applying the functional equation, we thus have that $$F_c(x)-\frac{a(x)}{b(x)}=0=F_c(x^c)-\frac{a(x^c)}{b(x^c)}=F_c(x)-\left(\frac{x^c}{c}+\frac{a(x^c)}{b(x^c)}\right),$$ so that \begin{equation}\label{abc}\frac{x^c}{c}+\frac{a(x^c)}{b(x^c)}=\frac{a(x)}{b(x)}.\end{equation} Now as functions, the righthand and lefthand sides of the equation in \eqref{abc} must have the same singularities. But $b(x^c)$ will have more zeros (counting multiplicity if needed) than $b(x)$ unless $b(x)$ is a constant, which is a contradiction. Thus $F_c(x)$ does not represent a rational function. In fact, we can now appeal to the following theorem, to give that $F_c(x)$ is transcendental over $\mathbb{C}(x)$.

\begin{theorem}[Nishioka, 1985]\label{KN} Suppose that $F(x)\in\mathbb{C}[[x]]$ satisfies one of the following for an integer $d>1$. \begin{align*} \textnormal{(i)}&\ \ F(x^d)=\phi(x,F(x)),\\ \textnormal{(ii)}&\ \ F(x)=\phi(x,F(x^d)),\end{align*} where $\phi(x,u)$ is a rational function in $x,u$ over $\mathbb{C}$. If $F(x)$ is algebraic over $\mathbb{C}(x)$, then $F(x)\in\mathbb{C}(x)$.
\end{theorem}

To prove the transcendence of the Stoneham numbers, we appeal to a classical result of Mahler \cite{Mahl1}, We record it here as taken from Nishioka's mongraph {\em Mahler Functions and Transcendence} \cite{Nish1}. 

\begin{theorem}[Mahler \cite{Mahl1}]\label{Mahler} Let $\mathbf{I}$ be the set of algebraic integers over $\mathbb{Q}$, $K$ be an algebraic number field, $\mathbf{I}_K=K\cap\mathbf{I},$ $f(x)\in K[[x]]$ with radius of convergence $R>0$ satisfying the functional equation for an integer $d>1$, $$f(x^d)=\frac{\sum_{i=0}^m a_i(x)f(x)^i}{\sum_{i=0}^m b_i(x)f(x)^i},\qquad m<d,\ a_i(x),b_i(x)\in \mathbf{I}_K[x],$$ and $\Delta(x):={\rm Res}(A,B)$ be the resultant of $A(u)=\sum_{i=0}^m a_i(x)u^i$ and $B(u)=\sum_{i=0}^m b_i(x)u^i$ as polynomials in $u$. If $f(x)$ is transcendental over $K(x)$ and $\xi$ is an algebraic number with $0<|\xi|<\min\{1,R\}$ and $\Delta(\xi^{d^n})\neq 0$ $(n\geqslant 0)$, then $f(\xi)$ is transcendental.
\end{theorem}

Since $F_c(x)$ is transcendental over $\mathbb{C}(x)$, $F_c(x)$ satisfies the functional equation \eqref{stonefunk}, and ${\rm Res}(cu-x^c,1)\neq 0$ for all $x$, we have the following corollary to Mahler's theorem.

\begin{corollary} Let $c\geqslant 2$ be an integer. The number $\sum_{n\geqslant 1}\frac{1}{c^n}\xi^{c^n}$ is transcendental for all algebraic numbers $\xi$ with $0<|\xi|<1.$ In particular, for all $b,c\geqslant 2$, the number Stoneham number $\alpha_{b,c}$ is transcendental.
\end{corollary} 

%%%%%%%%%%%%%%%%%%%%%%%%%%%%%%%%%%%%%%%%%%%%%%%%%%%%%%%%%%%%%%%%%%%%%%%%%%%%%
%%%%%%%%%%%%%%%%%%%%%%%%%%%%%%%%%%%%%%%%%%%%%%%%%%%%%%%%%%%%%%%%%%%%%%%%%%%%%
\bibliographystyle{line}
\bibliography{JAMS-paper}
\providecommand{\bysame}{\leavevmode\hbox to3em{\hrulefill}\thinspace}
\providecommand{\MR}{\relax\ifhmode\unskip\space\fi MR }
\providecommand{\MRhref}[2]{%
  \href{http://www.ams.org/mathscinet-getitem?mr=#1}{#2}
}
\providecommand{\href}[2]{#2}

\end{document}